\documentclass[10pt]{amsart}
\usepackage{amssymb}
\usepackage{amscd}

\def\today{\number\day\space\ifcase\month\or   January\or February\or
   March\or April\or May\or June\or   July\or August\or September\or
   October\or November\or December\fi\   \number\year}

\theoremstyle{definition}
\newtheorem{thm}{Theorem}[section]
\newtheorem{lem}[thm]{Lemma}
\newtheorem{prp}[thm]{Proposition}
\newtheorem{dfn}[thm]{Definition}
\newtheorem{cor}[thm]{Corollary}

\newtheorem{rmk}[thm]{Remark}

\newcommand{\beq}{\begin{equation}}
\newcommand{\eeq}{\end{equation}}
\newcommand{\beqr}{\begin{eqnarray*}}
\newcommand{\eeqr}{\end{eqnarray*}}
\newcommand{\bal}{\begin{align*}}
\newcommand{\eal}{\end{align*}}
\newcommand{\bei}{\begin{itemize}}
\newcommand{\eei}{\end{itemize}}

\newcommand{\C}{{\mathbb{C}}}
\newcommand{\N}{{\mathbb{N}}}

\newcommand{\ca}{C*-algebra}

\newcommand{\Index}{{\mathrm{Index}}}

\title[The Rohlin property for inclusions of $C^*$-albebras] {The Rohlin property for inclusions of $C^*$-algebras with a finite Watatani index}
\author{Hiroyuki Osaka$^*$}
\thanks{$^*$Research of the first author partially supported by the JSPS grant for Scientific Research No.20540220}
\address{ Department of Mathematical Sciences\\
  Ritsumeikan University\\ Kusatsu, Shiga, 520-2152  Japan}

\email[]{osaka@se.ritsumei.ac.jp}

\author{Kazunori Kodaka}

\address{Department of Mathematical Sciences, Faculty of Science,
Ryukyu University, \\
Nishihara-cho, Okinawa 903-0213, Japan}

\email[]{kodaka@math.u-ryukyu.ac.jp}

\author{Tamotsu Teruya}

\address{Department of Mathematical Sciences\\
  Ritsumeikan University\\ Kusatsu, Shiga, 520-2152  Japan}

\email[]{teruya@se.ritsumei.ac.jp}


\begin{document}
\maketitle

\begin{abstract}
We introduce notions of the Rohlin property and  the approximate 
representability
for inclusions of unital $C^*$-algebras.
We investigate a dual relation between the Rohlin property and 
the approximate representability.
We  prove that a number of classes of unital $C^*$-algebras are 
closed under inclusions with
the Rohlin property, including:
\begin{itemize}
 \item AF algebras, AI algebras, AT algebras, and related classes 
characterized by direct limit
 decomposition using semiprojective building blocks.
 \item $C^*$-algebras with  stable rank one.
 \item $C^*$-algebras with real rank zero.
\end{itemize}
\end{abstract}

\section{Introduction}
A.\ Kishimoto \cite{Kishimoto:77}, R.\ Herman and V.\ Jones 
\cite{HJ:82}, \cite{HJ:83}  investigated
a class of finite group actions with
what we presently call the Rohlin property.
After that a number of results for group actions of $C^*$-algebras 
with the Rohlin property
were found in the literature
(see \cite{Kishimoto:98a}, \cite{Kishimoto:98b} , \cite{Nkamura:00}, 
\cite{Izumi:Rohlin1}, \cite{Phillips:tracial}).

In \cite{Izumi:Rohlin1}, M.\ Izumi introduced the Rohlin property
and the approximate representability
for finite group actions.
He proved  that an action of a finite abelian group has the Rohlin property
if and only if its dual action is approximately representable.
We extend the notions of the Rohlin property and  the approximate 
representability for inclusions of unital $C^*$-algebras with
finite Watatani index in the sense of \cite{Watatani:index}.
We investigate a dual relation between the Rohlin property and the 
approximate representability.
We prove that an inclusion  has the Rohlin property if and only if 
its dual inclusion is approximately representable.
It contains that an action of a finite group has the Rohlin property 
if and only if its dual action is approximately representable
as a finite dimensional $C^*$-Hopf algebra action. Note that the dual 
action of an action of a non-commutative finite group is  not an action of some group
though it is an action of some finite dimensional commutative $C^*$-Hopf algebra. 

In \cite{OP:Rohlin}, H.\ Osaka and N.\ C.\ Phillips proved that 
crossed products by finite group actions with the
Rohlin property preserve various properties of $C^*$-algebras.
Since an action of a finite group with the Rohlin property  
is an outer action by \cite[Remark 1.4 and Lemma 1.5]{Phillips:tracial} and 
the crossed product algebra $A \rtimes _\alpha G$ and the fixed point algebra 
$A^{\alpha}$ by an outer action $\alpha$ of a finite group $G$
are Morita equivalent, we  can immediately see that fixed point algebras by 
finite group actions with the Rohlin property also preserve various properties of $C^*$-algebras by 
\cite{OP:Rohlin}. 
We extend their results  and prove that a number of classes of unital $C^*$-algebras are 
closed under inclusions with
the Rohlin property, including:
\begin{itemize}
 \item AF algebras, AI algebras, AT algebras, and related classes 
characterized by direct limit
 decomposition using semiprojective building blocks.
 \item $C^*$-algebras with  stable rank one.
 \item $C^*$-algebras with real rank zero.
\end{itemize}

This paper is organized as follows: In Section \ref{Preliminaries} we 
collect basic facts on
Watatani index theory for \ca s and finite group actions on \ca s with the 
Rohlin property.

In section \ref{Rohlin for E} we introduce the Rohlin property and 
the approximately representability
for conditional expectations and
deduce basic properties of conditional expectations possessing it.
Let $G$ be a finite group, $\alpha$ an action of $G$ on a unital simple 
$C^*$-algebra $A$, and $E$ the canonical conditional expectation from $A$ onto 
the fixed point algebra $A^\alpha$. We prove that 
$\alpha$ has the Rohlin property if and only if $E$ has the Rohlin property.
We prove that if an inclusion has a conditional expectation with the 
Rohlin property, then it is the unique
conditional expectation of its inclusion.
So the property that a conditional expectation has the Rohlin 
property is actually
a property of its inclusion.

In section \ref{Exam}
we construct an inclusion $A \supset P$ with the Rohlin property such that 
$P$ is not the fixed point algebra $A^\alpha$ for any finite group action $\alpha$ with the Rohlin property.

In section \ref{Rohlin for inclusion} we
prove that inclusions  with the
Rohlin property preserve various properties of $C^*$-algebras which 
generalize  results
of \cite{OP:Rohlin}.

\section{Preliminaries}\label{Preliminaries}

In this section we collect notations and basic facts which will be used in this paper.
\subsection{Index theory for \ca s}
\subsubsection{Watatani index for $C^*$-algebras}
We introduce an index in terms of a quasi-basis following Watatani \cite{Watatani:index}.
\begin{dfn}
Let $A \supset P$ be an  inclusion of unital \ca s\ with a conditional expectation $E$ from $A$ onto $P$.
\begin{enumerate}
 \item A {\it quasi-basis} for $E$ is a finite set $\{(u_i, v_i)\}_{i=1}^n \subset A \times A$ such that 
 for every $a \in A$, 
 $$
 a = \sum_{i=1}^nu_iE\left(v_i a\right)= \sum_{i=1}^n E\left(a u_i\right)v_i.
 $$
 \item When $\{(u_i, v_i)\}_{i=1}^n$ is a quasi-basis for $E$, we define $\Index E$ by 
 $$
 \Index E = \sum_{i=1}^n u_iv_i.
 $$
 When there is no quasi-basis, we write $\Index E = \infty$. $\Index E$ is called the 
 Watatani index of $E$.  
\end{enumerate}
\end{dfn}

\begin{rmk}\label{rmk:quasi} We give several remarks about the above definitions.
\begin{enumerate}
 \item $\Index E$ does not depend on the choice of the quasi-basis in the above formula, 
 and it is a central element of $A$ \cite[Proposition 1.2.8]{Watatani:index}.
 \item Once we know that there exists a quasi-basis, we can choose one of the form 
 $\{(w_i, w_i^*)\}_{i=1}^m$, which shows that $\Index E$ is a positive element \cite[Lemma 2.1.6]{Watatani:index}.
 \item By the above statements, if $A$ is a simple $C^*$-algebra, then $\Index E$ is a 
 positive scalar.
 \item Let $\{(u_i, v_i)\}_{i=1}^n$ be a quasi-basis for $E$. 
 If $A$ acts on a Hilbert space $\mathcal{H}$ faithfully, then we can define the map $E^{-1}$
 from $P' \cap B(\mathcal {H})$ to $A' \cap B(\mathcal{H})$ by $E^{-1}(x) = \sum_{i=1}^n u_i x v_i$ for $x$ in $P' \cap B(\mathcal{H})$. 
 In fact, 
 for any $x \in P'\cap B(\mathcal{H})$ and $a \in A$ 
 \begin{eqnarray*}
 E^{-1}(x)a&=& \sum_{i =1}^n u_i x v_i a \\
 &=&\sum_{i, j =1}^n u_i x E(v_i a u_j)v_j \\
 &=&\sum_{i, j =1}^n u_i E(v_i a u_j)x v_j \\
 &=&\sum_{j =1}^n  a u_jx v_j = aE^{-1}(x). \\
\end{eqnarray*}
\item If $\Index E < \infty$, then $E$ is faithful, that is, $E(x^*x) = 0$ implies $x=0$ for $x \in A$.
\end{enumerate}
\end{rmk}

\subsubsection{$C^*$-basic construction}

In this subsection, we recall Watatani's notion of the 
$C^*$-basic construction.

Let $E\colon A\to P$ be a faithful conditional expectation.
Then $A_{P}(=A)$ is
 a pre-Hilbert module over $P$ with a $P$-valued inner
product $$\langle x,y\rangle_P =E(x^{*}y), \ \ x, y \in A_{P}.$$
We denote by ${\mathcal E}_E$ and $\eta_E$ the Hilbert $P$-module completion of $A$ 
by the norm $\Vert x \Vert_P = \Vert \langle x, x \rangle_P\Vert^{\frac{1}{2}}$ for $x$ in $A$
and the natural inclusion map 
from $A$ into ${\mathcal E}_E$.  
Then ${\mathcal E}_E$ is a Hilbert $C^{*}$-module over $P$.
Since $E$ is faithful, the inclusion map $\eta_E$ from  $A$ to ${\mathcal E}_E$ is injective.
Let $L_{P}({\mathcal E}_E)$ be the set of all (right) $P$-module homomorphisms
$T\colon {\mathcal E}_E \to {\mathcal E}_E$ with an adjoint right $P$-module homomorphism
$T^{*}\colon {\mathcal E}_E \to {\mathcal E}_E$ such that $$\langle T\xi,\zeta
\rangle =
\langle \xi,T^{*}\zeta \rangle \ \ \ \xi, \zeta \in {\mathcal E}_E.$$
Then $L_{P}({\mathcal E}_E)$ is a $C^{*}$-algebra with the operator norm
$\|T\|=\sup\{\|T\xi \|:\|\xi \|=1\}.$ There is an injective
$*$-homomorphism $\lambda \colon A\to L_{P}({\mathcal E}_E)$ defined by
$$
\lambda(a)\eta_E(x)=\eta_E(ax)
$$
for $x\in A_{P}$ and  $a\in A$, so that $A$ can
be viewed as a
$C^{*}$-subalgebra of $L_{P}({\mathcal E}_E)$.
Note that the map $e_{P}\colon A_{P}\to A_{P}$
defined by 
$$
e_{P}\eta_E(x)=\eta_E(E(x)),\ \ x\in
A_{P}
$$
 is
bounded and thus it can be extended to a bounded linear operator, denoted
by $e_{P}$ again, on ${\mathcal E}_E$.
Then $e_{P}\in L_{P}({{\mathcal E}_E})$ and $e_{P}=e_{P}^{2}=e_{P}^{*}$; that
is, $e_{P}$ is a projection in $L_{P}({\mathcal E}_E)$.
A projection $e_P$ is called the {\em Jones projection} of $E$.

The {\sl (reduced) $C^{*}$-basic construction} is a $C^{*}$-subalgebra of
$L_{P}({\mathcal E}_E)$ defined to be
$$
C^{*}_r\langle A, e_{P}\rangle = \overline{ span \{\lambda (x)e_{P} \lambda (y) \in
L_{P}({{\mathcal E}_E}): x, \ y \in A \ \} }^{\|\cdot \|} 
$$

\begin{rmk}\label{rmk:b-const}
Watatani proved the following in \cite{Watatani:index}:
\begin{enumerate}
 \item $\Index E$ is finite if and only if $C^{*}_r\langle A, e_{P}\rangle$ has the identity 
 (equivalently $C^{*}_r\langle A, e_{P}\rangle = L_{P}({\mathcal E}_E)$) and there exists a constant 
 $c>0$ such that $E(x^*x) \geq cx^*x$ for $x \in A$, i.e., $\Vert x \Vert_P^2 \geq c\Vert x \Vert^2 $ 
 for $x$ in $A$ by \cite[Proposition 2.1.5]{Watatani:index}.
 Since $\Vert x \Vert \geq \Vert x \Vert_P$ for  $x$ in $A$, if $\Index E$ is finite, then ${\mathcal E}_E = A$.
 \item If $\Index E$ is finite, then each element $z$ in $C^{*}_r\langle A, e_{P}\rangle$ has a form 
 $$
 z = \sum_{i=1}^n \lambda(x_i) e_P \lambda(y_i)
 $$
 for some $x_i$ and $y_i $ in $A$.
 \item Let $C^{*}_{\rm max}\langle A, e_{P}\rangle$ be the unreduce $C^*$-basic construction defined in 
 Definition 2.2.5 of \cite{Watatani:index}, which has the certain universality (cf.(5)).
 If $\Index E$ is finite, then there is an isomorphism from 
 $C^{*}_r\langle A, e_{P}\rangle$ onto $C^{*}_{\rm max}\langle A, e_{P}\rangle$ (\cite[Proposition 2.2.9]{Watatani:index}).
 Therefore we  can identify $C^{*}_r\langle A, e_{P}\rangle$ with $C^{*}_{\rm max}\langle A, e_{P}\rangle$. 
 So we call it the $C^*$-{\it basic construction} and denote it by $C^{*}\langle A, e_{P}\rangle$. 
 Moreover 
 we identify $\lambda(A)$ with $A$ in $C^*\langle A, e_p\rangle (= C^{*}_r\langle A, e_{P}\rangle)$ 
and we denote 
$$
C^*\langle A, e_p\rangle = \{ \sum_{i=1}^n x_i e_P y_i : x_i, y_i \in A, n \in \N\}.
$$
\item If $\Index E$ is finite, then $\Index E$ is a central  invertible  element of $A$ and 
 there is the dual conditional expectation $\hat{E}$ from $C^{*}\langle A, e_{P}\rangle$ onto $A$ such that 
 $$
   \hat{E}(x e_P y) = (\Index E)^{-1}xy \quad \text{for} \ x, y \in A
 $$    
 by \cite[Proposition 2.3.2]{Watatani:index}. Moreover, $\hat{E}$ has a finite index and faithfulness.
\item Suppose that $\Index E$ is finite and $A$ acts on a Hilbert space $\mathcal{H}$ faithfully and $e$ is a projection on 
 $\mathcal{H}$ such that $eae =E(a)e$ for $a \in A$. If a map $P \ni x \mapsto xe \in B(\mathcal{H})$ 
 is injective, then there exists an isomorphism $\pi$ from the norm closure  of a linear span of $AeA$ to  
 $C^{*}\langle A, e_{P}\rangle$ such that $\pi(e) = e_P$ and $\pi(a) = a$ for $a \in A$ \cite[Proposition 2.2.11]{Watatani:index}.
\end{enumerate}
\end{rmk}

The next lemma is very useful.

\begin{lem}\label{lem:pd}
Let $A \supset P$ be an inclusion of unital \ca s\ with a conditional expectation $E$ from $A$ onto $P$.
If $\Index E$ is finite, then 
for each element $z$ in the $C^*$-basic construction $C^*\langle A, e_P \rangle $, there exists an element $a$ in $A$ such that 
$ze_P = ae_P$. In fact, 
$$
  ze_P = (\Index E)\hat{E}(ze_P)e_P.
$$
\end{lem} 
\begin{proof}
For each $z \in C^{*}\langle A, e_{P}\rangle$ there are  elements $x_i, y_i \in A$ such that $z = \sum_{i=1}^nx_i e_P y_i$.
Then $ze_P = \sum_i x_i e_P y_i e_P = \sum_i x_iE(y_i)e_P$, i.e., $ a = \sum_i x_iE(y_i) \in A$.
On the other hand, $\hat{E}(ze_P) = (\Index E)^{-1}\sum_i x_iE(y_i)$ by Remark~\ref{rmk:b-const}~(3) 
and hence we have $a = (\Index E)\hat{E}(xe_P)$.
\end{proof}

\begin{dfn}
	Let $A \supset P$ be a inclusion of unital $C^*$-algebras with a finite index and let $Q$ be 
	a $C^*$-subalgebra of $P$. $Q$ is said to be a {\it tunnel construction} for the inclusion $A\supset P$
	if $A$ is the basic construction for the inclusion $P \supset Q$.
\end{dfn}

In the factor case, a tunnel construction always exists for any index finite subfactor \cite{Jones:index}.
But an inclusion of $C^*$-algebras does not have a tunnel construction in general.
We shall give  a necessary and sufficient condition for 
an existence of a tunnel construction in the next proposition.

\begin{prp}\label{prp:tunnel}
Let $A \supset P$ be an inclusion of unital $C^*$-algebras and $E$ a conditional expectation from $A$ onto $P$ with
$\Index E < \infty $.
If there is a  projection $e \in A$ such that $E(e) = (\Index E)^{-1}$, 
then we have 
$$
ePe = Qe, \quad Q = P \cap \{e\}'
$$
In particular, if $e$ is a full projection, i.e., there are elements $x_i, y_i$ of $A$ such that 
$\sum_{i=1}^n x_i e y_i = 1$,  then 
$Q$ is a tunnel construction for $A \supset P$ such that $e$ is the Jones projection for $P \supset Q$.	
\end{prp} 
\begin{proof}
	Let $e_P$ be the Jones projection for $E$. Then $e_Pee_P = E(e)e_p = {(\Index E)}^{-1}e_P$.
	We shall prove that $ee_Pe = (\Index E)^{-1}e$. Put $f = (\Index E) ee_Pe$. Then it is easy to see that $f$ is a projection and 
	$f \leq e$. Let $\hat{E}$ be the dual conditional expectation of $E$. Then $\hat{E}(e -f) = 0$. And hence 
	we have $f=e$ by the faithfulness of $\hat{E}$ and $ee_Pe = (\Index E)^{-1}e$. 
	
	Let $F$ be a linear map on $P$ defined by $F(x) = (\Index E) E(exe)$. 
	We shall prove $exe = F(x)e = eF(x)$ for $x \in P$. Since $ee_Pe = (\Index E)^{-1} e$, we have for $x \in P$ 
	\begin{eqnarray*}
	eF(x)e_P &=& e(\Index E) E(exe)e_P \\
	&=& e(\Index E)  e_P(exe)e_P \\
        &=& (\Index E)  (ee_Pe)xee_P = exee_P 
	\end{eqnarray*}
	and hence $eF(x)e_P = exe e_P$. 
	Then using Remark~\ref{rmk:b-const}~(3) we have
        \begin{eqnarray*}
        eF(x) 
        &=& (\Index E) \hat{E}(eF(x)e_P) \\
        &=& (\Index E)\hat{E}(exee_P) = exe.
        \end{eqnarray*}
	Moreover $F(x)e = (eF(x^*))^* = (ex^*e)^* = exe$, and hence 
	$exe = F(x)e = eF(x)$.
	
	Let $Q$ be the $C^*$-subalgebra of $P$ defined by $Q = P \cap \{e\}'$.
	We saw $F(x) \in Q$ for any $x \in P$. Conversely if $x$ is an element of $Q$,  then 
	$F(x) = (\Index E) E(exe) = (\Index E) E(e)x = x$. Therefore 
        $F$ is a conditional expectation from $P$ onto $Q$ and  $ePe = Qe$.
	If $xe =0$ for some $x \in Q$, then $x = (\Index E)E(e)x = (\Index E)E(xe)=0$ and hence a map $Q \ni x \mapsto xe \in Qe$ 
        is injective. 
        By Remark \ref{rmk:b-const}~(5), 
        the norm closure of the linear span $\{ xey : x, y \in P\}$ is the basic construction for $P \supset Q$. 
	For any $a \in A$ 
	\begin{eqnarray*}
	ae &=& (\Index E)\hat{E}(e_P ae ) \\
        &=& (\Index E)^2\hat{E}(e_P a e e_P e) \\
        &=& (\Index E)^2\hat{E}(E(ae)e_Pe) = (\Index E) E(ae)e, 
	\end{eqnarray*}
	and hence $Ae = Pe$. Similarly, we have $eA = eP$. If $e$ is a full projection, then  
	\begin{eqnarray*}
        A &=&  \text{the linear span of }\{xe y : x, y \in A\} \\
        &=& \text{the linear span of }\{x e y : x, y \in P\}.
        \end{eqnarray*} 
        So $A$ is the  basic construction for the inclusion $P \supset Q$ with the Jones projection $e$. 
        It means that $Q$ is a tunnel construction for $A \supset P$.
	\end{proof}

\subsection{Finite group actions on $C^*$-algebras with the Rohlin property}

For a $C^*$-algebra $A$, we set 
\begin{eqnarray*}
c_0(A) &=& \{(a_n) \in l^\infty(\N, A): \lim\limits_{n \to \infty} \Vert a_n \Vert= 0\}  \\
 A^\infty &=&l^\infty(\N, A)/c_0(A).  
\end{eqnarray*}
We identify $A$ with the $C^*$-subalgebra of $A^\infty$ consisting of the equivalence classes of constant sequences and 
set
$$
A_\infty = A^\infty \cap A'.
$$
For an automorphism $\alpha \in {\rm Aut}(A)$, we denote by $\alpha^\infty$ and $\alpha_\infty$ the automorphisms of 
$A^\infty$ and $A_\infty$ induced by $\alpha$, respectively.

Izumi defined the Rohlin property for a finite group action in \cite[Definition 3.1]{Izumi:Rohlin1} as follows:

\begin{dfn}\label{def:group action}
Let $\alpha$ be an action of a finite group $G$ on a unital $C^*$-algebra $A$. 
$\alpha$ is said to have the {\it Rohlin property} if there exists a partition of unity 
$\{e_g\}_{g \in G} \subset A_\infty$ consisting of projections satisfying 
$$
(\alpha_g)_\infty(e_h) = e_{gh} \quad \text{for} \  g, h \in G.
$$
We call $\{e_g\}_{g\in G}$  Rohlin projections. 
\end{dfn}  

The next lemma is essentially contained in \cite[Lemma 1.5]{Phillips:tracial}. 
But we give a short proof of it for the self-contained.

\begin{lem}
Let $\alpha$ be an action of a finite group $G$ on a unital $C^*$-algebra $A$. 
If $\alpha$ has the Rohlin property, then $\alpha$ is an outer action.
\end{lem}

\begin{proof}
Suppose that $g$ is not the unit element of $G$.
If $\alpha_g$ is an inner automorphism ${\rm Ad} u$ for some unitary element $u$ in $A$, 
then $(\alpha_g)_\infty(e_h) = u e_h u^* = e_h$ for $h$ in $G$ since $e_h \in A_\infty = A' \cap A^\infty$.
Hence if $\alpha_g$ has the Rohlin property, then $\alpha_g$ is outer.
\end{proof}

Let $A \supset P$ be an inclusion of unital $C^*$-algebras.
For a conditional expectation $E$ from  $A$ onto $P$, we denote by $E^\infty$, the natural
conditional expectation from $A^\infty $ onto $ P^\infty$ induced by $E$.
If $E$ has a finite index with a quasi-basis $\{(u_i, v_i)\}_{i=1}^n$, then 
$E^\infty$ also has a finite index with a quasi-basis $\{(u_i, v_i)\}_{i=1}^n$
and $\Index (E^\infty) = \Index E$.

\begin{prp}\label{prp:group}
Let $\alpha$ be an action of a finite group $G$ on a unital $C^*$-algebra $A$ and 
$E$  the canonical conditional expectation from $A$ onto the fixed point algebra $P = A^{\alpha}$ defined by 
$$
E(x) = \frac{1}{\#G}\sum_{g\in G}\alpha_g(x) \quad \text{for} \ x \in A, 
$$
where $\#G$ is the order of $G$.
Then $\alpha$ has the Rohlin property if and only if there is a projection $e \in A_\infty$ such that 
$E^\infty (e) = \frac{1}{\#G}\cdot 1$, where $E^\infty$ is the conditional expectation from $A^\infty$ onto 
$P^\infty$ induced by $E$.
\end{prp}

\begin{proof}
Suppose that $\alpha$ has the Rohlin property with a partition of unity 
$\{e_g\}_{g \in G} \subset A_\infty$ consisting of projections satisfying 
$$
(\alpha_g)_\infty(e_h) = e_{gh} \quad \text{for} \  g, h \in G.
$$
Then 
$$
E^\infty(e_1) = \frac{1}{\#G}\sum_{g\in G}(\alpha_g)_\infty(e_1) = \frac{1}{\#G}\sum_{g\in G}e_g = \frac{1}{\#G}\cdot 1, 
$$
where $e_1$ is the projection in the partition of unity $\{e_g\}_{g\in G}$ which corresponds to the unit element of $G$.

Conversely, suppose that there is a projection $e \in A_\infty$ such that $E^\infty (e) = \frac{1}{\#G}\cdot 1$.
Define $e_g = (\alpha_g)_\infty(e)\in A_\infty$ for $g \in G$.
Then 
$$
\sum_{g\in G}e_g = \# G E(e) =1,
$$
i.e., $\{e_g\}_{g \in G} \subset A_\infty$ is a partition of unity. It is obvious that 
$(\alpha_g)_\infty(e_h) = e_{gh}$ for $g, h \in G$.
Hence $\alpha$ has the Rohlin property.
\end{proof}

%%%%%%%%%%%%%%%%%%%%%%%%%%%%%%%%%%%%%%%%%%%%%%%%%%%%%%%%%%%%%%%%%%%%%%%%%%%%%%%%%%%%%%%%%%%%%%%%%%%%%%%%%%%%%%%%%%%%%%%%%%%%%%%%

\section{Conditional expectations of unital $C^*$-algebras with the Rohlin property}\label{Rohlin for E}

\begin{dfn}
A conditional expectation $E$ of a unital  $C^*$-algebra $A$ with a finite index is said to have the {\it Rohlin property} 
if there exists a  projection $e \in A_\infty$ satisfying 
$$
E^\infty(e) = ({\Index}E)^{-1} \cdot 1
$$
and a map $A \ni x \mapsto xe$ is injective. We call $e$ a Rohlin projection.
\end{dfn}

\begin{prp}
Let $G$ be a finite group, $\alpha$  an action of $G$ on  a unital simple $C^*$-algebra $A$, and
 $E$  the canonical conditional expectation from $A$ onto the fixed point algebra $A^{\alpha}$.
Then $\alpha$ has the Rohlin property if and only if $E$ has the Rohlin property.
\end{prp}

\begin{proof}
Suppose that $\alpha$ has the Rohlin property. By \cite[Theorem 4.1 and Remark 4.6]{Jeong:saturated}, $\alpha$ is 
saturated and $\Index E$ is finite.
The simplicity of $A$ implies that the map $A \ni x \mapsto xe$ is injective. 
So we have that $E$ has the Rohlin property by Proposition \ref{prp:group}.
Conversely, if $E$ has the Rohlin property, then $\alpha$ has the Rohlin property by Proposition \ref{prp:group}.
\end{proof}

\begin{dfn}
A conditional expectation $E$ from  a unital $C^*$-algebra $A$ onto $P$ with a finite index 
is said to be {\it approximately representable} 
if there exists a projection $e \in P_\infty$ satisfying for any $x \in A$
$$
exe = E(x)e 
$$
and a map $P \ni x \mapsto xe $ is injective. 
\end{dfn}

\begin{prp}\label{prp:equality}
Let $A \supset P$ be an inclusion of unital $C^*$-algebras 
and $E$ a conditional expectation from 
$A$ onto $P$ with a finite index.   Let $B$ be the basic construction for $A \supset P$ 
and $\hat{E}$ the dual conditional expectation of $E$ from $B$ onto $A$.
Then 
\begin{enumerate}
 \item $E$ has the Rohlin property if and only if $\hat{E}$ is 
 approximately representable;
 \item $E$ is approximately representable if and only if 
 $\hat{E}$ has the Rohlin property.
\end{enumerate}
\end{prp}

\begin{proof} (1):
Let $e_P$ be the Jones projection for the inclusion $A \supset P$. 
Suppose that $E$ has the Rohlin property with a Rohlin projection $e \in A_\infty$.
Then 
$$
e_Pee_P = E^\infty (e)e_P ={(\Index E)}^{-1}e_P \quad \text{in } \ B^\infty.
$$
Let $f$ be an element in $B^\infty$ defined by $f= (\Index E) ee_Pe$.
It is easy to see that $f$ is a projection and $f \leq e$.
Since $\hat{E}^\infty(e -f) = e - (\Index E) e\hat{E}(e_P)e = e-e =0$, we have 
$f = e $ by the faithfulness of $\hat{E}^\infty$, and hence $ee_Pe = {(\Index E)}^{-1}e$.
Since $\hat{E}$ is determined by $\hat{E}(xe_Py) = {(\Index E)}^{-1}xy$ for $x$ and  $y$ in $A$, 
we have for any element $z$ in $B$
$$
eze = \hat{E}(z)e, 
$$
and the map $A \ni x \mapsto xe \in Ae$ for $x \in A$ is injective by the definition of the Rohlin property.
Therefore $\hat{E}$ is approximately representable.

Conversely, suppose that $\hat{E}$ is approximately representable with a projection $e \in A_\infty$
satisfying that $eze = \hat{E}(z)e$ for any $z$ in $B$ and 
a map $A \ni x \mapsto xe \in Ae$ for $x \in A$ is injective.
Then we have 
$$
ee_Pe = \hat{E}(e_P)e = {(\Index E)}^{-1}e.
$$
Define an element $w$ in $B^\infty$ by $w = (\Index E) e_Pe e_P$.
Then 
$$
\hat{E}^\infty(e_P-w)e = e(e_P - w)e = ee_Pe - (\Index E) ee_P(ee_Pe) = 0.
$$
By the faithfulness of $\hat{E}^\infty$ and the injectivity of the map $x \mapsto xe$ for $x \in A$, 
we have 
$w=e_P$, i.e., $e_Pee_P = (\Index E)^{-1}e_P$.
Since $ E^\infty(e)e_P = e_Pee_P = (\Index E)^{-1}e_P$, 
we have $E^\infty(e) = {(\Index E)}^{-1}$. 
Hence 
$E$ has the Rohlin property.
\par
(2): Suppose that $E$ is approximately representable with a projection $e \in P_\infty$ 
satisfying that $exe =E(x)e$ for any $x$ in $A$. Let 
$\{(u_i, v_i)\}_{i=1}^n$ be a quasi-basis for $E$.
Define an element $f$ in $B^\infty$ by 
$$
f = \sum_{i=1}^nu_i e e_P v_i  \left(=(E^\infty)^{-1}(ee_P)\right) .
$$ 
It is easy to see that $f$ is a projection and commutes with elements of $A$ by Remark~\ref{rmk:quasi}~(4).
$f$ also commutes with $e_P$. In fact, since $ee_P = e_Pe$, we have
\begin{eqnarray*}
fe_P &=& \sum_{i=1}^n u_i ee_Pv_i e_P \\
 &=&   \sum_{i=1}^n u_i eE(v_i)e_P = \sum_{i=1}^n u_iE(v_i) ee_P = ee_P 
\end{eqnarray*}
and
\begin{eqnarray*}
e_Pf &=& \sum_{i=1}^n e_Pu_i ee_Pv_i \\
 &=& \sum_{i=1}^n e_Pu_i e_Pev_i = \sum_{i=1}^n ee_PE(u_i)v_i  = ee_P.
 \end{eqnarray*}
Therefore $f$ is an element in $B' \cap B^\infty = B_\infty$ since $B$ is generated by $A$ and $e_P$.
By Remark~\ref{rmk:b-const}~(5), there exists an isomorphim $\pi$ from $B$ onto $C^*\langle A, e\rangle$ such that
$\pi(e_P) = e$ and $\pi(a) = a$ for $a \in A$. So we have 
$$
(E^\infty)^{-1}(e) = \sum_{i=1}^n u_i \pi(e_P)v_i = \pi\left(\sum_{i=1}^n u_i e_P v_i\right) =  1
$$
and 
\begin{eqnarray*}
\hat{E}^\infty(f) &=& \sum_{i=1}^n u_i e \hat{E}(e_P)v_i \\
 &=& (\Index E)^{-1} (E^\infty)^{-1}(e) = (\Index E)^{-1}\cdot 1.
\end{eqnarray*}
Suppose that $xf= 0$ for some  $x\in B$.
Let $C^*\langle A, e_P, e\rangle$ be a $C^*$-algebra generated by $A$, $e_P$, and $e$ in $B^\infty$. 
Let $\varphi$ be an automorphism on $C^*\langle A, e_P, e\rangle$ defined by $\varphi(a) = a$ for $a \in A$, 
$\varphi(e_P) =e$, and $\varphi(e) = e_P$. 
Since $\varphi(f) = f$ and $\varphi(x) = \pi(x)$ for $ x \in B$, we have $0 = \varphi(xf) = \pi(x)f$.
Since $\pi(x) \in A^\infty$, we have $0=\hat{E}^\infty(\pi(x)f)= \pi(x)\hat{E}^\infty(f) = {(\Index E)}^{-1}\pi(x)$.
Hence $x =0$, i.e., the map $B \ni x \mapsto xf \in Bf$ is an injective map.
Therefore $\hat{E}$ has the Rohlin property with a projection $f \in B_\infty$.

Conversely, suppose that $\hat{E}$ has the Rohlin property with a Rohlin projection $f \in B_\infty$.
Define an element $e$ in $A^\infty$ by $ e = (\Index E) \hat{E}^\infty(fe_P)$.
Using the fact that $fe_P = e_P f = (\Index E) \hat{E}^\infty (fe_P)e_P$ by Lemma \ref{lem:pd}, 
we have
\begin{eqnarray*}
 e^2&=& (\Index E)^2\hat{E}^\infty (fe_P)\hat{E}^\infty (fe_P) \\
 &=& (\Index E) \hat{E}^\infty ((\Index E) \hat{E}^\infty (fe_P)e_Pf) \\
 &=&(\Index E) \hat{E}^\infty(fe_P) =e, 
\end{eqnarray*}
and hence $e$ is a projection in $A^\infty$.
There exists an element $x$ in $P^\infty$ such that $fe_P = xe_P$ since 
$$
fe_P = e_Pfe_P \in e_P B^\infty e_P = \left(e_PBe_P\right)^\infty =(Pe_P)^\infty= P^\infty e_P. 
$$
Then we have
$$
 e = (\Index E) \hat{E}^\infty(fe_P) = (\Index E) \hat{E}^\infty(xe_P) = (\Index E) x \hat{E}(e_P) =x.
$$
Hence $e \in P^\infty$ and $fe_P = ee_P$. 
Since $e$ commutes with any element in $P$, we have  $e \in P_\infty$.
For any $a \in A$
\begin{eqnarray*}
eae &= & (\Index E)^2\hat{E}^\infty(fe_P)a \hat{E}^\infty(fe_P)\\
 &=& (\Index E)^2\hat{E}^\infty(\hat{E}^\infty(fe_P)a fe_P)\\
 &=& (\Index E)\hat{E}^\infty((\Index E)\hat{E}^\infty(fe_Pa) e_Pf)\\
 &=& (\Index E)\hat{E}^\infty(fe_Pa e_Pf) \quad (\text{by Lemma~\ref{lem:pd}})\\
 &=& (\Index E)\hat{E}^\infty(fE(a) e_Pf)\\
 &=& E(a)(\Index E)\hat{E}^\infty(f e_P)=E(a)e.
\end{eqnarray*}
Suppose that $xe =0$ for some  $x \in P$.  Since  $ee_P = fe_p$, 
$xfe_P = (xe_P)f =0$ and $xe_P = 0$ by the injectivity of the map $B \ni y \mapsto yf \in Bf$. Hence 
$x =0$ by the injectivity of the map $P \ni x \mapsto xe_P \in Pe_P$. Therefore $E$  is approximately representable.
\end{proof}

\begin{prp}\label{prp:app}
Let $A \supset P$ be an inclusion of unital \ca s and 
$E$  a conditional expectation from  $A$ onto $P$ with a finite index.
If $E$ is approximately representable, then $P' \cap A \subset P$. 
\end{prp}

\begin{proof}
Let $e$ be a projection in $P_\infty$ such that $exe = E(x)e$ for $x$ in $A$.
If $x$ is an element in $P'\cap A$, then $x$ also commutes with $e$.
Hence $xe = exe = E(x)e$. 
Since a map $y \ni P \mapsto ye \in Pe$ is injective,    
there is an isomorphism $\pi$ from 
the basic construction $C^*\langle A, e_P\rangle$ onto $C^*\langle A, e\rangle$ by Remark~\ref{rmk:b-const}~(5).
Then 
\begin{eqnarray*}
xe_P &=& x\pi(e) = \pi(xe) = \pi(exe) \\
&=& \pi(E(x)e)= E(x)\pi(e) = E(x)e_P. 
\end{eqnarray*}
Therefore $x = E(x) \in P$.
\end{proof}

\begin{rmk}
If a conditional expectation $E\colon A \to P$ is approximately representable, 
then $E$ is the unique conditional expectation from $A$ onto $P$ by \cite[Corollary 1.4.3]{Watatani:index}.
In other words, the property that $E$ is approximately representable
is actually a property of the inclusion $A \supset P$. 
By Proposition \ref{prp:equality}, the property that $E$ has the Rohlin property
is actually a property of the inclusion $A \supset P$. 
So we call $A \supset P$ an inclusion with the Rohlin property.
\end{rmk}

When an inclusion $A \supset P$ has a finite index, if $P$ is simple, then $A$ is a finite direct sum of simple closed two-sided ideals  
by \cite[Theorem 3.3]{Izumi:inclusion}. 
Therefore the above propositions immediately implies the following:

\begin{cor}\label{cor:simple}
Let $A \supset P$ be an inclusion of unital \ca s and 
$E$  a conditional expectation from  $A$ onto $P$ with a finite index.
If $E$ is approximately representable and $P$ is simple, 
then $A \supset P$ is an irreducible inclusion, i.e., 
$P' \cap A \cong \C$ and $A$ is simple.
\end{cor}

\begin{proof}
By Proposition~\ref{prp:app}, we have $P' \cap A \subset P' \cap P \cong \C$ by the simplicity of $P$.
On the other hand, $A' \cap A \subset P' \cap A \cong \C$ and hence $A$ is simple by \cite[Theorem 3.3]{Izumi:inclusion}.
\end{proof}

\begin{cor}\label{cor:irreducible}
Let $E$ be a conditional expectation from  a unital \ca \ $A$ onto $P$ with a finite index.
If $E$ has the Rohlin property and $A$ is simple, then $A \supset P$ is an irreducible inclusion and $P$ is simple.
\end{cor}

\begin{proof}
By Proposition \ref{prp:equality}, the dual conditional expectation $\hat{E}$ of $E$ is approximately representable.
Therefore the inclusion $C^*\langle A, e_P\rangle \supset A$ is irreducible and 
$C^*\langle A, e_P\rangle$ is simple by Corollary~\ref{cor:simple}.  Since $A' \cap C^*\langle A, e_P\rangle$ is isomorphic to 
$P' \cap A$ as linear spaces (see the proof of Proposition 3.11 of \cite{JOPT:can}), $A \supset P$ is irreducible.
The simplicity of $P$ comes from  \cite[Corollary 2.2.14]{Watatani:index}.
\end{proof}

%%%%%%%%%%%%%%%%%%%%%%%%%%%%%%%%%%%%%%%%%%%%%%%%%%%%%%%%%%%%%%%%%%%%

\section{Examples}\label{Exam}

In this section we shall construct examples of inclusions of $C^*$-algebras with the Rohlin property.

\begin{prp}
Let $Q$ be a unital simple $C^*$-algebra and $\alpha$ an action of a finite group $G$ on $Q$ with Rohlin projections 
$\{e_g\}_{g \in G}$. Let $H$ be a subgroup of $G$ and $P \subset A$ the inclusion of fixed point algebras $Q^G \subset Q^H$.
Then $P \subset A$ has the Rohlin property with a Rohlin projection $e_H = \sum_{h \in H} e_h$.
\end{prp}

\begin{proof}
Let $E$ be the canonical conditional expectation from $Q$ onto $P$ defined by 
$$
E(x) = \frac{1}{|G|}\sum_{g\in G} \alpha_g(x), \qquad x \in Q.
$$
Since $E(e_g) = \dfrac{1}{|G|}$ for any $g \in G$, we have
$$
E(e_H) = \sum_{h \in H} E(e_h) = \frac{|H|}{|G|}.
$$
Let $E_H$ be the conditional expectation from $Q$ onto $A$ defined by 
$$
E_H(x) = \frac{1}{|H|}\sum_{h\in H} \alpha_h(x), \qquad x \in Q.
$$
It is easy to see that $E_H^\infty (e_H) = e_H \in A^\infty$ since $\alpha^\infty_g(e_h) = \alpha^\infty(e_{gh})$.
By the definition of Rohlin projections we have $e_g \in  Q' \cap Q^\infty \subset A'\cap Q^\infty$ and  hence $e_H \in A' \cap A^\infty$.
Let $F$ be the conditional expectation from $A$ onto $P$ defined by $F(x) = E(x)$ for $x$ in $A$.
Since $\Index E = |G|$ and $\Index E_H = |H|$ and $E = F \circ E_H$,   $\Index F = \dfrac{|G|}{|H|}$ by the multiplicity of the index
(\cite[Proposition 1.7.1]{Watatani:index}).
We have $F(e_H) = \dfrac{1}{\Index F}$ and the inclusion $A \supset P$ has  the Rohlin property.
\end{proof}

\begin{rmk}
With the above notation if $H$ is not a normal subgroup of $G$, then it is easy to see that 
the depth of $A \supset P$ is not two (cf. \ \cite[Teorem 4.6]{Te:normal}).
In particular, $P$ is not the fixed point algebra $A^\alpha$ for any finite group action $\alpha$ with the Rohlin property.
\end{rmk}

%%%%%%%%%%%%%%%%%%%%%%%%%%%%%%%%%%%%%%%%%%%%%%%%%%%%%%%%%%%%%%%%%%%%%%%%%%%%%%%%%%%%%%%%%%%%%%%%%%%%%%%%%%%%%%%%%%%%%%%%%%%%%%

\section{Inclusions of $C^*$-algebras with the Rohlin property}\label{Rohlin for inclusion}

In \cite{OP:Rohlin}, Osaka and Phillips proved that crossed products by finite group actions with the 
Rohlin property preserve various properties of $C^*$-algebras.
In this section, we extend their result and prove that inclusions with the Rohlin property preserve various properties of 
$C^*$-algebras.

\begin{lem}\label{lem:isoD-A}
Let $A \supset P$ be an inclusion of unital \ca s and 
$E$  a conditional expectation from $A$ onto $P$ with a finite index.
If $E$ has the Rohlin property with a Rohlin projection $e \in A_\infty$, then for any $x \in A^\infty$ there exists 
the unique element $y$ of $P^\infty$ such that $xe = ye$. 
\end{lem}

\begin{proof}
Let $e_P$  be the Jones projection for the inclusion $A \supset P$.
By the proof of Proposition \ref{prp:tunnel}, we have $ee_Pe = (\Index E)^{-1}e$. 
Therefore for any element $x$ in $A^\infty$
\begin{eqnarray*}
 xe&=& (\Index E)\hat{E}^\infty(e_P xe )\\
 &=& (\Index E)^2\hat{E}^\infty(e_P x e e_P e)\\
 &= &(\Index E)^2\hat{E}^\infty(E^\infty(xe)e_Pe) = (\Index E) E^\infty(xe)e,
\end{eqnarray*}
where $\hat{E}$ is the dual conditional expectation for $E$. Put $y = (\Index E) E^\infty(xe) \in P^\infty$.
 Then we have
$x e = y e$. 

Suppose that $ye = ze$ for $y, z \in P^\infty$. 
Then 
$$
z = (\Index E) E^\infty(e)z = (\Index E) E^\infty(ez)=(\Index E) E^\infty(ey) = y.
$$ 
Therefore we obtain the uniqueness of $y$.
\end{proof}

\begin{rmk}\label{rmk:D}
Let $D$ be the $C^*$-subalgebra of $P^\infty$ defined by $D = \{e\}' \cap P^\infty$.
By Proposition \ref{prp:tunnel}, we have $eP^\infty e = De$.
Moreover, $eA^\infty e = De$ by the above lemma. 
\end{rmk}

\begin{thm}\label{thm:AF}
Let $A \supset P$ be an inclusion of separable unital \ca s and 
$E$  a conditional expectation from $A$ onto $P$ with a finite index.
If $A$ is an AF algebra and $E$ has the Rohlin property, then $P$ is also an AF algebra.
\end{thm}

\begin{proof}
We shall prove that for every finite set $S \subset P$ and every $\varepsilon>0$, there is a finite dimensional 
$C^*$-subalgebra $Q$ of $P$ such that every element of $S$ is within $\varepsilon$ of an element of $Q$.
Since $A$ is an AF algebra and $S \subset P \subset A$, there is a finite dimensional 
$C^*$-subalgebra $R$ of $A$ such that every element of $S$ is within $\frac{\varepsilon}{4}$ of an element of $R$.
Let $\{e_{ij}^{(r)}\}$ be a system of matrix units of $R \cong \bigoplus {\mathbb M}_{n_r}$. 
Since $e$ commutes with each element of $A$ and a map $A \ni x \to xe \in Ae$ is injective, 
$\{e_{ij}^{(r)}e\}$ is also a system of matrix units of type $R$. 
Since $Ae = eAe \subset eA^\infty e$, , there are elements $p_{ij}^{(r)}$ of 
$D = \{e\}' \cap P^\infty$ such that 
$e_{ij}^{(r)}e = p_{ij}^{(r)}e$ by  Lemma \ref{lem:isoD-A} and Remark \ref{rmk:D}.
By the uniqueness of $p_{ij}^{(k)}$, $\{p_{ij}^{(r)}\}$ is a system of matrix units of type $R$.
For every $(i,j,r)$, let $(p_{ijk}^{(r)})_{k=1}^\infty$ be a sequence of $P$ such that  
$p_{ij}^{(r)} = (p_{ijk}^{(r)})_{k=1}^\infty + c_0(A)$. 
For every $x \in S$, let $x^R = \sum_{i,j,r}x_{ij}^{(r)}e_{ij}^{(r)} \in R, \ x_{ij}^{(r)}\in \C$ such that 
$\Vert x -x^R \Vert <\frac{\varepsilon}{4}$.
Let $x^\infty$ be the element of $P^\infty$ defined by $x^\infty= \sum_{i,j,r}x_{ij}^{(r)}p_{ij}^{(r)}$.
For each $k \in \N$, we define $x^\infty_k$ as 
the element of $P$  by $x^\infty_k= \sum_{i,j,r}x_{ij}^{(r)}p_{ijk}^{(r)}$. 
Since $\Vert x -x^R \Vert = \Vert xe - x^Re \Vert = \Vert xe - x^\infty e \Vert
= \Vert x - x^\infty \Vert$ by the injectivity of a map $D \ni x \to xe \in D e$, 
we have 
$\Vert x - x^\infty \Vert <\frac{\varepsilon}{4}$ and hence 
$$
\limsup_{k \to \infty}  \Vert x - x^\infty_k \Vert <\frac{\varepsilon}{4}.
$$
Since $\{p_{ij}^{(r)}\}$ is a system of matrix units of type $R$,
we have  
$$
\limsup_{k \to \infty}\Vert p_{ijk}^{(r)} p_{m n k}^{(s)} - \delta_{r s}\delta_{jm}p_{ink}^{(r)}\Vert = 0.
$$
Choose $\delta>0$ according to \cite[Lemma 1.10]{Glimm:certain}
for $R$ and for $\varepsilon/(2\dim(R))$.
There is 
$k_0 \in \N$ such that $\{p_{ijk_0}^{(r)}\}$ is a set of approximate matrix units of type $R$ within $\delta$ in $P$
and $\Vert x - x^{\infty}_{k_0}\Vert < \frac{\varepsilon}{2}$ for every $x$ in $S$.
Then there is a set $\{f_{ij}^{(r)}\}$ of exact matrix units of type $R$ in $P$ with 
 $\Vert f_{ij}^{(r)}-p_{ijk_0}^{(r)}\Vert < \varepsilon/(2\dim(R))$.
 Put $x_0 = \sum_{i,j,r}x_{ij}^{(r)}f_{ij}^{(r)} \in P$. Since $\Vert x^\infty_{k_0}-x_0 \Vert < \frac{\varepsilon}{2}$, 
we have
\begin{eqnarray*}
 \Vert x-x_0\Vert &\leq & \Vert x-x^\infty_{k_0} \Vert+ \Vert x^\infty_{k_0} - x_0\Vert   \\
 &\leq & \frac{\varepsilon}{2} +\frac{\varepsilon}{2}= \varepsilon.
\end{eqnarray*}
So we can choose a finite dimensional $C^*$-algebra $Q$ generated by $\{f^{(r)}_{ij}\} \subset P$.
Therefore $P$ is an AF algebra.
\end{proof}

Following \cite{OP:Rohlin}, we introduce several notations to describe   
the local approximate characterizations of the classes of direct limit algebras constructed using 
some common families of semiprojective building blocks.

\begin{dfn}
Let $\mathcal{C}$ be a class of separable unital $C^*$-algebaras. Then $\mathcal{C}$ is 
{\it finitely saturated} if the following closure conditions hold:
\begin{enumerate}
 \item If $A \in \mathcal{C}$ and $B\cong A$, then $B \in \mathcal{C}$.
 \item If $A_1, A_2, \dots , A_n \in \mathcal{C}$ then $\bigoplus _{k=1}^n A_k \in \mathcal{C}$.
 \item If $A \in \mathcal{C}$ and $n \in \N$, then $M_n(A) \in \mathcal{C}$.
 \item If $A \in \mathcal{C}$ and $p \in A$ is a nonzero projection, then $pAp \in \mathcal{C}$.
\end{enumerate}
Moreover, the {\it finite saturation} of a class $\mathcal{C}$ is the smallest finitely saturated class which contains 
$\mathcal{C}$.
\end{dfn}

\begin{dfn}
Let $\mathcal{C}$ be a class of separable unital $C^*$-algebras. We say that 
$\mathcal{C}$ is {\it flexible} if :
\begin{enumerate}
 \item For every $A \in \mathcal{C}$, every $n \in \N$, and every nonzero projection $p \in M_n(A)$, 
 the corner $pM_n(A)p$ is semiprojective in the sense of Definition 14.1.3 of \cite{Loring:lifting}, 
 and is finitely generated.
 \item For every $A \in \mathcal{C}$ and every ideal $I \subset A$, there is an increasing sequence 
 $I_0 \subset I_1 \subset \cdots$ of ideals in $A$ such that $\overline{\cup_{n=0}^\infty I_n} =I$, 
 and such that for every $n$ the $C^*$-algebra $A/I_n$ is in the finite saturation of $\mathcal{C}$.
\end{enumerate}
\end{dfn}

\begin{dfn}
Let $\mathcal{C}$ be a class of separable unital $C^*$-algebras. A {\it unital local $\mathcal{C}$-algebra} 
is a separable unital $C^*$-algebra such that for every finite set $S \subset A$ and every $\varepsilon >0$, 
there are a $C^*$-algebra $B$ in the finite saturation of $\mathcal{C}$ and a unital homomorphism 
$\varphi: B \to A$ such that ${\rm dist}(a, \varphi(B)) < \varepsilon$ for all $a \in S$.
\end{dfn}

\begin{thm}\label{thm:local C}
Let $\mathcal{C}$ be any flexible class of separable unital $C^*$-algebras.
Let $A \supset P$ be a finite index inclusion with the Rohlin property. 
If $A$ is a unital local $\mathcal{C}$-algebra, then $P$ is also a unital local $\mathcal{C}$-algebra. 
\end{thm}

\begin{proof}
We shall prove that for every finite set $S \subset P$ and every $\varepsilon>0$, there are a 
$C^*$-algebra $Q$ in the finite saturation of $\mathcal{C}$ and a unital homomorphism 
$\varphi:Q \to P$ such that $S$ is within $\varepsilon$ of an element of $\varphi(Q)$.

Since $A$ is a unital local $\mathcal{C}$-algebra, for finite set $S \subset P \subset A$ and 
$\frac{\varepsilon}{2}>0$, there is a 
$C^*$-algebra $Q$ in the finite saturation of $\mathcal{C}$ and a unital homomorphism 
$\rho:Q \to A$ such that $S$ is within $\frac{\varepsilon}{2}$ of an element of $\rho(Q)$.
As in the proof of Theorem \ref{thm:AF}, $Ae = eAe \subset eA^\infty e = De$, where $e$ is a Rohlin projection 
for the inclusion $A \supset P$ and $D = \{e\}' \cap P^\infty$. 
By Lemma \ref{lem:isoD-A}, 
we can define a map $\beta:A \to D$ such that $ae = \beta(a)e$ for $a \in A$. 
It is easy to see that $\beta$ is a unital injective homomorphism and $\beta(x) = x$ for 
$x \in P$.
So we can define a unital homomorphism 
$\varphi^\infty : Q \to P^\infty$ by $\varphi^\infty (q) = \beta(\rho(q))$.
Since $\Vert x - a \Vert = \Vert \beta(x) - \beta(a)\Vert = \Vert x - \beta(a)\Vert$ for $x \in S$ and $a \in A$, 
we know that $S$ is within $\frac{\varepsilon}{2}$ of an element of $\varphi^\infty(Q)$.

For $n \in \N$, let $I_n$ be an ideal of $l^\infty(\N, P)$ defined by 
$$
I_n = \{(a_k)_{k=1}^\infty \in l^\infty(\N, P) : a_k = 0 \ \text{for} \ k >n \} .
$$
Then $\{I_n\}$ is an increasing chain of ideals in $l^\infty(\N, P)$ and 
$c_0(P) = \overline{\bigcup I_n}$.
Since $\mathcal{C}$ is flexible, $Q$ has the semiprojectivity. By the definition of the semiprojectivity 
(Definition 14.1.3 in \cite{Loring:lifting}), there exist $n \in \N$ and $\bar{\varphi}$ so that 
the diagram
\begin{eqnarray*}
 & & l^\infty(\N, P)/ I_n\\
 & \overset{\bar{\varphi}}{\nearrow} & \qquad \downarrow \\
 Q&\quad \overset{\varphi^\infty}\longrightarrow & \quad P^\infty
\end{eqnarray*}
commutes.  
For each $k \in \N$, let $\varphi_k$ be a map from $Q$ to $P$ so that $\bar{\varphi}(q) = (\varphi_k(q))_{k=1}^\infty + I_n$
for $q \in Q$.
By the above commutative diagram, we have $\varphi^\infty(q) = (\varphi_k(q))_{k=1}^\infty + c_0(P)$ and 
$\varphi_k$ is a homomorphism for  $k >n$.
For $x \in S$, we can choose $q_x \in Q$  such that 
$\Vert x -\varphi^\infty(q_x) \Vert < \frac{\varepsilon}{2}$. Then we have 
$$
\limsup_{k \to \infty}\Vert x - \varphi_k(q_x) \Vert < \frac{\varepsilon}{2}. 
$$ 
Since $S$ is a finite set, there exists $k_0 > n$ such that 
$\Vert x - \varphi_{k_0}(q_x) \Vert < \varepsilon$ for every $x$ in $S$.
Therefore $P$ is  a unital local $\mathcal{C}$-algebra.
\end{proof}

We have the following result.

\begin{cor}
Let $A \supset P$ be an inclusion of separable unital $C^*$-algebras with the Rohlin property.
\begin{enumerate}
 \item If $A$ is a unital AI algebra, as defined in Example 2.2 in \cite{OP:Rohlin}, then $P$ is a unital AI algebra.
 \item If $A$ is a unital AT algebra, as defined in Example 2.3 in \cite{OP:Rohlin}, then $P$ is a unital AT algebra.
 \item If $A$ is a unital AD algebra, as defined in Example 2.4 in \cite{OP:Rohlin}, then $P$ is a unital AD algebra. 
 \item If $A$ is a unital countable direct limit of one dimensional noncommutative CW complexes (Definition 2.5 in \cite{OP:Rohlin}), 
 then so is $P$.
\end{enumerate}
\end{cor}

\begin{proof}
Since the relevant classes are flexible by \cite{OP:Rohlin}, we may apply Theorem \ref{thm:local C}.
\end{proof}

The notion of topological stable rank for a $C^*$-algebra $A$, denoted by tsr($A$), was 
introduced by Rieffel, which generalized the concept of the dimension of a topological space 
\cite{Rf1}.  A unital $C^*$-algebra $A$ has topological stable rank one if 
the set of invertible elements of $A$  is dense in $A$.
We have the following result.

\begin{thm}\label{thm:str}
Let $A \supset P$ be an inclusion of  unital $C^*$-algebras with the Rohlin property.
If ${\rm tsr}(A) = 1$, then ${\rm tsr}(P) = 1$.
\end{thm}

\begin{proof}
We shall prove that for every element $x$ of $P$ and every $\varepsilon>0$, there is an invertible element $y$ of $P$ such that 
$\Vert x-y \Vert < \varepsilon$.

So fix $x \in P$ and $\varepsilon>0$. Since ${\rm tsr}(A)=1$ and $x \in P \subset A$, there is an invertible element $a \in A$ 
such that $\Vert x-a \Vert < \frac{\varepsilon}{2}$.
Let $\beta : A \to D = \{e\}' \cap P^\infty$ be the injective homomorphism defined in the proof of Theorem \ref{thm:local C}, where 
$e$ is the Rohlin projection in $A_\infty$.
Since $\beta(x) = x$, we have $\Vert x-\beta(a) \Vert < \frac{\varepsilon}{2}$.
Let $(a_n)$ be a sequence of elements in $P$ so that $\beta(a) = (a_n) + c_0(P)$. 
Then we have 
$$
\limsup_{n \to \infty}\Vert x-a_n \Vert < \frac{\varepsilon}{2}.
$$
Since $a$ is invertible element in $A$, so is $\beta(a) = (a_n) + c_0(P)$ in $P^\infty$.
Therefore there is $k \in \N$ such that $a_k$ is an invertible element in $P$ and 
$\Vert x-a_k \Vert < \varepsilon$.
\end{proof}

The theory of real rank for $C^*$-algebra, developed by Brown and Pedersen \cite{BP91}, formally resembles the theory of 
topological stable 
rank, but there are important differences under the surface. 
On the other hand, the real rank zero property is one of the most significant properties that a 
$C^*$-algebra can have. 
A unital $C^*$-algebra $A$ has real rank zero if the set of invertible self-adjoint elements of $A$ 
is dense in the set of self-adjoint elements of $A$.
We have, then,  the following theorem. Since its proof is very similar to 
the proof of Theorem \ref{thm:str}, we omit it.

\begin{thm}
Let $A \supset P$ be an inclusion of  unital $C^*$-algebras with the Rohlin property.
If $A$ has real rank zero, then P has real rank zero.
\end{thm}

\noindent {\bf Acknowledgement} {The authors would like to thank the referee for his useful comments and constructive suggestion.}

\end{document}